\documentclass{amsart}%
\usepackage{amsmath}
\usepackage{amsfonts}
\usepackage{amssymb}
\usepackage{graphicx}%
\newtheorem{lemma}{Lemma}[section]
\newtheorem{theorem}[lemma]{Theorem}
\newtheorem{remark}[lemma]{Remark}

\newtheorem{coro}[lemma]{Corollary}
\newtheorem{definition}[lemma]{Definition}
\newtheorem{example}[lemma]{Example}

\title[Two-sided Remotely Almost Periodic Solutions \ldots]
{Two-sided Remotely Almost Periodic Solutions of Ordinary Differential Equations in Banach Spaces}

\author{David Cheban}
\address[D. Cheban]{State University of Moldova\\
Faculty of Mathematics and Informatics\\
Laboratory of Fundamental and Applied Mathematics\\
A. Mateevich Street 60\\
MD--2009 Chi\c{s}in\u{a}u, Moldova} \email[D.
Cheban]{david.ceban@usm.md, davidcheban@yahoo.com}

\date{\today}

\begin{document}

\begin{abstract}
The aim of this paper is studying the two-sided remotely almost
periodic solutions of ordinary differential equations in Banach
spaces of the form $x'=A(t)x+f(t)+F(t,x)$ with two-sided remotely
almost periodic coefficients if the linear equation $x'=A(t)x$
satisfies the condition of exponential trichotomy and nonlinearity
$F$ is "small".
\end{abstract}

\maketitle

\section{Introduction}\label{Sec1}

The aim of this talk is studying the remotely almost periodic on
the real axis $\mathbb R$ solutions of differential equations.
This study continues the author's series of works devoted to the
study of remotely almost periodic motions of dynamical systems and
solutions of differential equations
\cite{Che_2009},\cite{Che_2024}-\cite{Che_2024.2}.

The notion of remotely almost periodicity on the real axis
$\mathbb R$ for the scalar functions was introduced and studied by
Sarason D. \cite{Sar_1984}. Remotely almost periodic functions on
the semi-axis $\mathbb R_{+}$ with the values in the Banach space
were introduced and studied by Ruess W. M. and Summers W. H.
\cite{RS_1986} (see also \cite{Che_2024,Che_2024_01} and
\cite{Che_2024_1}-\cite{Che_2024.2}). Remotely almost periodic
functions on the real axis with the values in Banach spaces were
introduced and studied by Baskakov A. G. \cite{Bas_2013}. He calls
theses functions "almost periodic at infinity". Remotely almost
periodic (on the real-axis $\mathbb R$) solutions of ordinary
differential equations with remotely almost periodic coefficient
were studied by Maulen C., Castillo S., Kostic M. and Pinto M.
\cite{MCKP_2021}.

The paper is organized as follows. In the second section we
collect some known notions and facts about remotely almost
periodic motions of dynamical systems and remotely almost periodic
functions. In the third section we collect some facts about the
linear hyperbolic (satisfying the condition of exponential
dichotomy) cocycle over semi-group dynamical system on the
semi-axis $\mathbb R_{+}$. The fourth section is dedicated to the
study the remotely almost periodic on the semi-axis $\mathbb
R_{+}$ solutions of linear differential equations wit remotely
almost periodic coefficients. In the fifth and sixth sections we
study the problem of existence at least one two-side remotely
almost periodic solution of linear and semi-linear differential
equations with two-sided remotely almost periodic (in time)
coefficients.

\section{Preliminary}\label{Sec2}

%\subsection{Remotely Almost Periodic Motions}\label{Sec2.1}

Let $X$ and $Y$ be two complete metric spaces, $\mathbb
R:=(-\infty,+\infty)$ and $\mathbb T \in \{\mathbb R,\ \mathbb
R_{+}\}$ and $(X,\mathbb R_{+},\pi)$ (respectively, $(Y,\mathbb R,
\sigma )$) be an autonomous one-sided (respectively, two-sided)
dynamical system on $X$ (respectively, $Y$).

Let $(X,\mathbb T,\pi)$ be a dynamical system and
$\pi(t,x)=\pi^{t}x=xt$.

\begin{definition}\label{defSP1} A point $x\in X$ (respectively, a motion $\pi(t,x)$) is
said to be:
\begin{enumerate}
\item[-] stationary, if $\pi(t,x)=x$ for any $t\in \mathbb T$;
\item[-] $\tau$-periodic ($\tau >0$ and $\tau \in \mathbb T$), if
$\pi(\tau,x)=x$; \item[-] asymptotically stationary (respectively,
asymptotically $\tau$-periodic), if there exists a stationary
(respectively, $\tau$-periodic) point $p\in X$ such that
\begin{equation}\label{eqAP1*}
\lim\limits_{t\to \infty}\rho(\pi(t,x),\pi(t,p))=0.\nonumber
\end{equation}
\end{enumerate}
\end{definition}

\begin{theorem}\label{thAAP1}\cite[Ch.I]{Che_2009} A point $x\in X$ is asymptotically $\tau$-periodic if and
only if the sequences $\{\pi(k\tau,x)\}_{k=0}^{\infty}$ converges.
\end{theorem}

\begin{definition}\label{defLS1} A point $\tilde{x}\in X$ is said
to be $\omega$-limit for $x\in X$ if there exists a sequence
$\{t_k\}\subset \mathbb S_{+}$ such that $t_k\to +\infty$ and
$\pi(t_k,x)\to \tilde{x}$ as $k\to \infty$.
\end{definition}

Denote by $\omega_{x}$ the set of all $\omega$-limit points of
$x\in X$.

\begin{definition}\label{defSAP1}  We will call
a point $x\in X$ (respectively, a motion $\pi(t,x)$) remotely
$\tau$-periodic ($\tau\in \mathbb T$ and $\tau
>0\widetilde{}$) if
\begin{equation}\label{eqSAP_1}
 \lim\limits_{t\to
+\infty}\rho(\pi(t+\tau,x),\pi(t,x))=0 .\nonumber
\end{equation}
\end{definition}

%\begin{remark}\label{remS1.0} The motions of dynamical systems possessing the property
%(\ref{eqSAP_1}) was studied in the works of K. Cryszka
%\cite{Gry_2018} and A. Pelczar \cite{Pel_1985}.
%\end{remark}

\begin{definition}\label{defLS02} A point $x$ is called Lagrange
stable (respectively, positively Lagrange stable), if the
trajectory $\Sigma_{x}:=\{\pi(t,x)|\ t\in \mathbb T\}$
(respectively, semi-trajectory $\Sigma^{+}_{x}:=\{\pi(t,x)|\ t\ge
0\}$) is a precompact subset of $X$.
\end{definition}

\begin{theorem}\label{th1.3.9}\cite[Ch.I]{Che_2020} Let $x\in X$ be
positively Lagrange stable and $\tau\in\mathbb T\ (\tau >0)$. Then
the following statements are equivalent:
\begin{enumerate}
\item[a.] the motion $\pi(t,x)$ is remotely $\tau$-periodic;
\item[b.] any point $p\in\omega_{x}$ is $\tau$-periodic.
\end{enumerate}
\end{theorem}

\begin{definition}\label{defSAP2} A point $x$ (respectively, a
motion $\pi(t,x)$) is said to be remotely stationary, if it is
remotely $\tau$-periodic for any $\tau \in \mathbb T$.
\end{definition}

\begin{coro}\label{corSAP1} Let $x\in X$ be positively Lagrange stable. Then the following
statements are equivalent:
\begin{enumerate}
\item[a.] the motion $\pi(t,x)$ is remotely stationary; \item[b.]
any point $p\in\omega_{x}$ is stationary.
\end{enumerate}
\end{coro}
\begin{proof} This statement follows directly from Definition \ref{defSAP2}
and Theorem \ref{th1.3.9}.
\end{proof}

\begin{definition}\label{defRAP0} A subset $A\subseteq \mathbb
T$ is said to be relatively dense in $\mathbb T$ if there exists a
positive number $l\in \mathbb T$ such that $[a,a+l]\bigcap A \not=
\emptyset$ for any $a\in \mathbb T$, where $[a,a+l]:=\{x\in
\mathbb T|\ a\le x\le a+l\}$.
\end{definition}

%\begin{remark}\label{remDT1} For any $\tau >0$ ($\tau \in \mathbb
%T$) the set $A:=\{k\tau |\ k\in \mathbb Z\}\bigcap \mathbb T$ is
%relatively dense in $\mathbb T$.
%\end{remark}

\begin{definition}\label{defAP1} A point $x\in X$ of dynamical
system $(X,\mathbb T,\pi)$ is said to be:
\begin{enumerate}
\item almost periodic if for any $\varepsilon >0$ the set
\begin{equation}\label{eqAP1}
\mathcal P(\varepsilon,p):=\{\tau \in \mathbb T|\
\rho(\pi(t+\tau,p),\pi(t,p))<\varepsilon \ \ \mbox{for any}\ t\in
\mathbb T\}\nonumber
\end{equation}
is relatively dense in $\mathbb T$; \item asymptotically
stationary (respectively, asymptotically $\tau$-periodic,
asymptotically almost periodic or positively asymptotically
Poisson stable) if there exists a stationary (respectively,
$\tau$-periodic, almost periodic or positively Poisson stable)
point $p\in X$ such that
\begin{equation}\label{eqAP3}
\lim\limits_{t\to \infty}\rho(\pi(t,x),\pi(t,p))=0.\nonumber
\end{equation}
\end{enumerate}
\end{definition}

\begin{definition}\label{defRAP1} A point $x\in X$ (respectively,
a motion $\pi(t,x)$) is said to be remotely almost periodic
\cite{RS_1986} if for arbitrary positive number $\varepsilon$
there exists a relatively dense subset $\mathcal
P(\varepsilon,x)\subseteq \mathbb T$ such that for any $\tau \in
\mathcal P(\varepsilon,x)$ there exists a number
$L(\varepsilon,x,\tau)>0$ for which we have
\begin{equation}\label{eqRAP1}
\rho(\pi(t+\tau,x),\pi(t,x))<\varepsilon \nonumber
\end{equation}
for any $t\ge L(\varepsilon,x,\tau)$.
\end{definition}

\begin{remark}\label{remAP1} Every almost periodic point $x\in X$
is remotely almost periodic.
\end{remark}

\begin{lemma}\label{lRAP_01} \cite{Che_2024_1} Every remotely $\tau$-periodic (respectively,
remotely stationary) point $x$ of the dynamical system $(X,\mathbb
T,\pi)$ is remotely almost periodic.
\end{lemma}

Denote by $\mathfrak L_{x}^{\pm \infty}:=\{\{t_n\}\subset \mathbb
T|\ \{\pi(t_n,x)\}$ converges and $t_n\to +\infty$ as $n\to
\pm\infty$ $\}$ and $\mathfrak L_{x}:=\mathfrak
L_{x}^{-\infty}\bigcup \mathfrak L_{x}^{+\infty}$.

\begin{definition}\label{defRAP01} Let $(X,\mathbb T,\pi)$ and $(Y,\mathbb
T,\sigma)$ be two dynamical systems. A point $x\in X$ is said to
be
\begin{enumerate}
\item positively (respectively, negatively) remotely comparable by
the character of recurrence with the point $y\in Y$ if $\mathfrak
L_{y}^{+\infty}\subseteq \mathfrak L_{x}^{+\infty}$ (respectively,
$\mathfrak L_{y}^{-\infty}\subseteq \mathfrak L_{x}^{-\infty}$);
\item remotely comparable (or two-sided comparable) with the point
$y\in Y$ if $\mathfrak L_{y}\subseteq \mathfrak L_{x}$.
\end{enumerate}
\end{definition}

\begin{theorem}\label{thRAP4.0} \cite[Ch.I]{Che_2020} Let $y\in Y$ be asymptotically stationary
(respectively, asymptotically $\tau$-periodic or asymptotically
almost periodic) point. If the point $x\in X$ is positively
remotely comparable by the character of recurrence with the point
$y$, then the point $x$ is also asymptotically stationary
(respectively, asymptotically $\tau$-periodic or asymptotically
almost periodic).
\end{theorem}

\begin{theorem}\label{thRAP4} \cite{Che_2024_1} Let $y\in Y$ be Lagrange stable and remotely stationary
(respectively, remotely $\tau$-periodic or remotely almost
periodic) point. If the point $x\in X$ is remotely comparable by
the character of recurrence with the point $y$, then the point $x$
is also remotely stationary (respectively, remotely
$\tau$-periodic or remotely almost periodic).
\end{theorem}

%\subsection{Remotely Almost Periodic Functions}\label{Sec2.2}

Denote by $C(\mathbb T,X)$ the space of all continuous functions
$\varphi :\mathbb T\to X$ equipped with the distance
\begin{equation*}\label{eqD_01}
\beta(\varphi,\psi):=\sup\limits_{L>0}\min\{\max\limits_{|t|\le
L,\ t\in \mathbb T}\rho(\varphi(t),\psi(t)),L^{-1}\}.
\end{equation*}
The space $(C(\mathbb T,X),d)$ is a complete metric space (see,
for example, \cite[ChI]{Che_2020}).

Let $h\in \mathbb T$, $\varphi \in \mathbb C(\mathbb T,X)$ and
$\varphi^{h}$ be the $h$-translation, i.e.,
$\varphi^{h}(t):=\varphi(t+h)$ for any $t\in \mathbb T$. Denote by
$\sigma_{h}$ the mapping from $C(\mathbb T,X)$ into itself defined
by equality $\sigma_{h}\varphi :=\varphi^{h}$ for any $\varphi \in
C(\mathbb T,X)$. Note that $\sigma_{0}=Id_{C(\mathbb T,X)}$ and
$\sigma_{h_1}\sigma_{h_2}=\sigma_{h_1+h_2}$ for any $h_1,h_2\in
\mathbb T$.

\begin{lemma}\label{lAPF1}\cite[Ch.I]{Che_2020} The mapping $\sigma :\mathbb T\times C(\mathbb T,X)\to C(\mathbb
T,X)$ defined by $\sigma(h,\varphi)=\sigma_{h}\varphi$ for any
$(h,\varphi)\in \mathbb T\times C(\mathbb T,\mathfrak B)$ is
continuous.
\end{lemma}

\begin{coro}\label{corAPF1} The triplet $(C(\mathbb T,X),\mathbb
T,\sigma)$ is a dynamical system (shift dynamical system or
Bebutov's dynamical system).
\end{coro}

Let $\mathfrak B$ be a  Banach space over the field $P$
($P=\mathbb R$ or $\mathbb C$) with the norm $|\cdot|$,
$\rho(u,v):=|u-v|$ ($u,v\in \mathfrak B$) and $\tau\in \mathbb T$
be a positive number. Denote by $C_{0}(\mathbb T,\mathfrak
B)):=\{\varphi \in C(\mathbb T,\mathfrak B)$ such that
$\lim\limits_{t\to +\infty}|\varphi(t)|=0\}$ and $C_{\tau}(\mathbb
T,\mathfrak B):=\{\varphi \in C(\mathbb T,\mathfrak B)|\
\varphi(t+\tau)=\varphi(t)$ for any $t\in \mathbb T\}$.

\begin{definition}\label{defAPF1} Let $\tau \in \mathbb T$ and $\tau >0$. A function $\varphi \in C(\mathbb
T,X)$is said to be:
\begin{enumerate}
\item asymptotically $\tau$-periodic (respectively, asymptotically
stationary) if there exist $p\in C_{\tau}(\mathbb T,\mathfrak B)$
and $r\in C_{0}(\mathbb T,\mathfrak B)$ such that
$\varphi(t)=p(t)+r(t)$ for any $t\in \mathbb T$; \item remotely
$\tau$-periodic \cite{HPT_2008,Kal_2010} (respectively, remotely
stationary) if
\begin{equation}\label{eqAPF2}
\lim\limits_{t\to +\infty}\rho(\varphi(t+\tau),\varphi(t))=0
\end{equation}
(respectively, remotely $\tau$-periodic for any $\tau
>0$);
\item remotely almost periodic
\cite{Bas_2013,Bas_2015,BSS_2019,RS_1986,Sar_1984} if for every
$\varepsilon
>0$ there exists a relatively dense subset $\mathcal
P(\varepsilon,\varphi)$ such that for any $\tau \in \mathcal
P(\varepsilon,\varphi)$ we have a positive number
$L(\varepsilon,\varphi,\tau)$ so that
\begin{equation}\label{eqAPF2.1}
\rho(\varphi(t+\tau),\varphi(t))<\varepsilon \nonumber
\end{equation}
for any $t\ge L(\varepsilon,\varphi,\tau)$.
\end{enumerate}
\end{definition}

\begin{remark}\label{remAPF_02} 1. The notion of remotely almost periodicity
on the real axis $\mathbb R$ for the scalar functions was
introduced and studied by Sarason D. \cite{Sar_1984}.

2. Remotely almost periodic functions on the semi-axis $\mathbb
R_{+}$ with the values in the Banach space were introduced and
studied by Ruess W. M. and Summers W. H. \cite{RS_1986}.

3. Remotely almost periodic functions on the real axis with the
values in the Banach spaces were introduced and studied by
Baskakov A. G. \cite{Bas_2013}. He calls theses functions "almost
periodic at infinity".

4. The functions with the property (\ref{eqAPF2}) in the work
\cite{HPT_2008} (respectively, in the work \cite{Kal_2010}) is
called $S$-asymptotically $\tau$-periodic (respectively,
$\tau$-periodic at the infinity).
\end{remark}

\begin{remark}\label{remAPF1} 1. Every remotely
$\tau$-periodic function is remotely almost periodic.

2. Every asymptotically $\tau$-periodic (respectively,
asymptotically stationary) function $\varphi \in C(\mathbb
T,\mathfrak B)$ is remotely $\tau$-periodic
\cite{HPT_2008,Kal_2010} (respectively, remotely stationary).
\end{remark}

\begin{definition}\label{defAPF02} A function $\varphi \in C(\mathbb
T,X)$ is said to be Lagrange stable if the motion
$\sigma(t,\varphi)$ is so in the shift dynamical system
$(C(\mathbb T,X),\mathbb T,\sigma)$.
\end{definition}

\begin{lemma}\label{lAPF02}  {\rm(\cite{Sel_1971,sib})}  A function $\varphi \in C(\mathbb T,X)$
is Lagrange stable if and only if the following conditions are
fulfilled:
\begin{enumerate}
\item the set $\varphi(\mathbb T):=\{\varphi(t)|\ t\in \mathbb
T\}$ is precompact in $X$; \item the function $\varphi$ is
uniformly continuous on $\mathbb T$.
\end{enumerate}
\end{lemma}

\begin{lemma}\label{lAPF_3}\cite{Che_2023wp} Let $\varphi \in C(\mathbb T,X)$ and $\omega \in
\mathbb T$ be a positive number. The following statements are
equivalent:
\begin{enumerate}
\item
\begin{equation}\label{eqAPF_3}
\lim\limits_{t\to
+\infty}\rho(\varphi(t+\tau),\varphi(t))=0;\nonumber
\end{equation}
\item
\begin{equation}\label{eqAPF04}
\lim\limits_{t\to
+\infty}\beta(\sigma(t+\tau,\varphi),\sigma(t,\varphi))=0
.\nonumber
\end{equation}
\end{enumerate}
\end{lemma}

\begin{lemma}\label{lAPF_03} Let $\varphi \in C(\mathbb T,X)$ be a Lagrange stable function. The function $\varphi$ is
remotely $\tau$-periodic if and only if its $\omega$-limit set
$\omega_{\varphi}$ consists of a $\tau$-periodic functions.
\end{lemma}
\begin{proof} This statement directly follows from Lemma
\ref{lAPF_3} and Theorem \ref{th1.3.9}.
\end{proof}

\begin{coro}\label{corAPF_03} Let $\varphi \in C(\mathbb T,X)$ be a Lagrange stable function. The function $\varphi$ is
remotely stationary if and only if its $\omega$-limit set
$\omega_{\varphi}$ consists of a stationary functions.
\end{coro}
\begin{proof} This statement directly follows from Lemma
\ref{lAPF_03} and Corollary \ref{corSAP1}.
\end{proof}

\begin{lemma}\label{lAPF3.1} \cite{Che_2024.1} Let $\varphi \in C(\mathbb T,X)$. The
following statement are equivalent:
\begin{enumerate}
\item[(a)] the motion $\sigma(t,\varphi)$ generated by the
function $\varphi$ in the shift dynamical system $(C(\mathbb
T,X),\mathbb T,\sigma)$ is remotely almost periodic (respectively,
remotely $\tau$-periodic or remotely stationary); \item[(b)] the
function $\varphi$ is remotely almost periodic (respectively,
remotely $\tau$-periodic or remotely stationary).
\end{enumerate}
\end{lemma}

\begin{definition}\label{defAAP1} A function $\varphi \in C(\mathbb T,\mathfrak
B)$ is said to be asymptotically stationary (respectively,
asymptotically $\tau$-periodic, asymptotically almost periodic or
positively asymptotically Poisson stable) if there exist functions
$p,r\in C(\mathbb T,\mathfrak B)$ such that
\begin{enumerate}
\item $\varphi(t)=p(t)+r(t)$ for any $t\in \mathbb T$; \item $r\in
C_{0}(\mathbb T,\mathfrak B)$ and $p$ is stationary (respectively,
$\tau$-periodic, almost periodic or positively Poisson stable
($p\in \omega_{p}$)).
\end{enumerate}
\end{definition}

\begin{lemma}\label{lAPP1} \cite[Ch.I]{Che_2009} The following statements are
equivalent:
\begin{enumerate}
\item the function $\varphi \in C(\mathbb T,X)$ is asymptotically
stationary (respectively, asymptotically $\tau$-periodic,
asymptotically almost periodic or positively asymptotically
Poisson stable); \item the motion $\sigma(t,\varphi)$ of shift
dynamical system $C(\mathbb T,\mathfrak B),\mathbb T,\sigma)$ is
asymptotically stationary (respectively, asymptotically
$\tau$-periodic, asymptotically almost periodic or positively
asymptotically Poisson stable).
\end{enumerate}
\end{lemma}

%\begin{theorem}\label{th1.1RAP} If the function $\varphi \in C(\mathbb T,X)$ is
%remotely stationary (respectively, remotely $\tau$-periodic or
%remotely almost periodic) and positively asymptotically Poisson
%%stable, then it is asymptotically stationary (respectively,
%asymptotically $\tau$-periodic or asymptotically almost periodic).
%\end{theorem}
%\begin{proof} This statement follows from Lemmas \ref{lAPF3.1}, \ref{lAPP1} and  Theorems \ref{thAAP1}, \ref{thAAP_1}.
%\end{proof}

Let $\varphi \in C(\mathbb R,X)$. Denote by $\mathfrak
L_{\varphi}^{\pm \infty}:=\{\{t_n\}|$ such that $t_n\to \pm\infty$
and $\{\varphi^{t_n}\}$ converges in $C(\mathbb R,X)$ as $n\to
\infty\}$ and $\mathfrak L_{\varphi}:=\mathfrak
L_{\varphi}^{-\infty}\bigcup \mathfrak L_{\varphi}^{+\infty}$.

\begin{definition}\label{defCF1} A function $\varphi \in C(\mathbb
R,X)$ is said to be
\begin{enumerate}
\item positively (respectively, negatively) remotely comparable
(comparable at the infinity) by character of recurrence with the
given function $\psi \in C(\mathbb R,Y)$ if $\mathfrak
L_{\psi}^{+\infty}\subseteq \mathfrak L_{\varphi}^{+\infty}$
(respectively, $\mathfrak L_{\psi}^{+\infty}\subseteq \mathfrak
L_{\varphi}^{+\infty}$); \item  remotely comparable (comparable at
the infinity) by character of recurrence with the given function
$\psi \in C(\mathbb R,Y)$ if $\mathfrak L_{\psi}\subseteq
\mathfrak L_{\varphi}$.
\end{enumerate}
\end{definition}

\begin{theorem}\label{thRAP4.F0} Let $\psi \in C(\mathbb T,Y)$ be an asymptotically stationary
(respectively, asymptotically $\tau$-periodic or asymptotically
almost periodic) function. If the function $\varphi \in C(\mathbb
T,X)$ is positively remotely comparable by the character of
recurrence with the function $\psi$, then the function $\varphi$
is also asymptotically stationary (respectively, asymptotically
$\tau$-periodic or asymptotically almost periodic).
\end{theorem}

\begin{theorem}\label{thRAPF4} Let $\psi \in C(\mathbb T,Y)$ be a Lagrange stable and remotely stationary
(respectively, remotely $\tau$-periodic or remotely almost
periodic) function. If the function $\varphi \in C(\mathbb T,X)$
is positively remotely comparable by the character of recurrence
with the function $\psi$, then the function $\varphi$ is also
remotely stationary (respectively, remotely $\tau$-periodic or
remotely almost periodic).
\end{theorem}

Theorem \ref{thRAP4.F0} (respectively, Theorem \ref{thRAPF4})
directly follows from Theorem \ref{thRAP4.0} (respectively,
Theorem \ref{thRAP4}).

Denote by $C(\mathbb T\times \mathfrak B, \mathfrak B)$ the space
of all continuous functions $f:\mathbb T \times \mathfrak B \to
\mathfrak B$ equipped with the compact-open topology.

Denote by $(C(\mathbb T\times \mathfrak B,\mathfrak B), \mathbb
T,\sigma)$ the shift dynamical system on the space $C(\mathbb
T\times \mathfrak B,\mathfrak B)$ (see, for example,
\cite[Ch.I]{Che_2015}), i.e., $\sigma(h,f):=f^{h}$ and
$f^{h}(t,x):=f(t+h,x)$ for any $(t,x)\in \mathbb T\times \mathfrak
B$.

\begin{definition}\label{defD1} A function $f\in C(\mathbb T\times \mathfrak B,\mathfrak B)$
is said to be remotely $\tau$-periodic (respectively, positively
Lagrange stable and so on) in $t\in \mathbb T$ uniformly with
respect to $x$ on every compact subset from $\mathfrak B$ if the
motion $\sigma(t,f)$ (defined by function $f$) of the dynamical
system $(C(\mathbb T\times \mathfrak B,\mathfrak B),\mathbb
T,\sigma)$ is remotely $\tau$-periodic (respectively, positively
Lagrange stable and so on).
\end{definition}

\section{Hyperbolic Linear Nonautonomous Dynamical Systems}\label{Sec3}

%\subsection{Linear nonautonomous dynamical systems}\label{Sec3.1}

Let $(\mathfrak B, |\cdot |)$ be a Banach space with the norm
$|\cdot|$, $\mathbb T\in \{\mathbb R_{+},\mathbb R\}$ and $\langle
\mathfrak B, \varphi, (Y,\mathbb R, \sigma)\rangle$ (or shortly
$\varphi$) be a linear cocycle over dynamical system $(Y,\mathbb
R,\sigma)$ with the fibre $\mathfrak B$, i.e., $\varphi$ is a
continuous mapping from $\mathbb T\times \mathfrak B \times Y$
into $\mathfrak B$ satisfying the following conditions:
\begin{enumerate}
\item $\varphi(0,u,y)=u$ for any $u\in\mathfrak B$ and $y\in Y$;
\item
$\varphi(t+\tau,u,y)=\varphi(t,\varphi(\tau,u,y),\sigma(\tau,y))$
for any $t,\tau\in\mathbb T$, $u\in \mathfrak B$ and $y\in Y$;
\item for any $(t,y)\in \mathbb T\times Y$ the mapping
$\varphi(t,\cdot,y):\mathfrak B\mapsto \mathfrak B$ is linear.
\end{enumerate}

Denote by $[\mathfrak B]$ the Banach space of any linear bounded
operators $A$ acting on the space $\mathfrak B$ equipped with the
operator norm $||A||:=\sup\limits_{|x|\le1}|Ax|$\index{$||A||$}.

\begin{example}\label{exHom1}  Let $Y$ be a complete metric space,
$(Y,\mathbb R,\sigma)$ be a dynamical system on $Y$ and $
[\mathfrak B]$ be the Banach space of linear bounded operators
acting on Banach space $ \mathfrak B $ and $f\in C(Y,\mathfrak
B)$. Consider the following linear (respectively, linear
nonhomogeneous) differential equation
\begin{equation}\label{eqLDE1}
x'=A(\sigma(t,y))x,
\end{equation}
where $A\in C(Y,[\mathfrak B])$. Note that the following
conditions are fulfilled (see, for example,
\cite[Ch.III]{DK_1970}) for equation (\ref{eqLDE1}):
\begin{enumerate}
\item[a.] for any $ v \in \mathfrak B $ and $ y\in Y $ the
equation (\ref{eqLDE1}) has exactly one solution $ \varphi
(t,v,y)$ with the condition $ \varphi (0, v, y ) = v$ is
fulfilled; \item
$\varphi(t+\tau,v,y)=\varphi(t,\varphi(\tau,v,y),\sigma(\tau,y))$
for any $(t,\tau)\in \mathbb T\times \mathbb R$ and $(v,y)\in
\mathfrak B\times Y$; \item[c.] the mapping $ \varphi : (t,u,y)
\to \varphi (t,u,y) $ is continuous in the topology of $\mathbb T
\times \mathfrak B \times Y$; \item $\varphi(t,\alpha u+\beta v,
y)=\alpha\varphi(t,u,y)+\beta \varphi(t,v,y)$ for any $t\in
\mathbb T$, $u,v\in \mathbb B$ and $\alpha,\beta \in P$.
\end{enumerate}

The equation (\ref{eqLDE1}) generates a linear cocycle $\langle
\mathfrak B, \varphi, (Y,\mathbb R, \sigma)\rangle$ over dynamical
system $(Y,\mathbb R,\sigma)$ with the fiber $\mathfrak B$.
\end{example}

\begin{example}\label{exHom2}{\rm
Consider a linear nonhomogeneous differential equation
\begin{equation}\label{eqLDE2}
x'=A(t)x,
\end{equation}
where $A\in C(\mathbb R,[\mathfrak B])$. Along this equation
(\ref{eqLDE2}) consider its $H$-class, i.e., the following family
of equations
\begin{equation}\label{eqLDE3}
x'=B(t)x,
\end{equation} where $B\in
H(A):=\overline{\{A^{h}|\ h\in \mathbb R\}}$, by bar the closure
in the space $C(\mathbb R,[\mathfrak B])$ is denoted and
$A^{h}(t):=A(t+h)$ for any $t\in \mathbb R$. Note that the
following conditions are fulfilled (see, for example,
\cite[Ch.III]{DK_1970}) for the equation (\ref{eqLDE2}) and its
$H$-class (\ref{eqLDE3}):
\begin{enumerate}
\item[a.] for any $ u \in \mathfrak B $ and $ B\in H(A)$ the
equation (\ref{eqLDE3}) has exactly one solution $ \varphi
(t,u,B)$ and the condition $ \varphi (0, u, B ) = v$ is fulfilled;
\item[b.]
$\varphi(t+\tau,v,B)=\varphi(t,\varphi(\tau,v,B),B^{\tau})$ for
any $t,\tau \in \mathbb T$ and $(v,B)\in \mathfrak B\times H(A)$;
\item[c.] the mapping $ \varphi : (t,u,B ) \to \varphi (t,u,B ) $
is continuous in the topology of $\mathbb R_{+} \times \mathfrak B
\times H(A)$; \item[d.] $\varphi(t,\alpha u+\beta v,
B)=\alpha\varphi(t,u,B)+\beta \varphi(t,v,B)$ for any $t\in
\mathbb T$, $u,v\in \mathbb B$, $B\in H(A)$ and $\alpha,\beta \in
P$.
\end{enumerate}

Thus, every linear differential equation of form (\ref{eqLDE2})
(and its $H$-class (\ref{eqLDE3})) generates a linear cocycle
$\langle \mathfrak B, \psi, (H(A,f),\mathbb T,\sigma)\rangle$ over
dynamical system $(H(A,f),\mathbb R,\sigma)$ with the fibre
$\mathfrak B$.}
\end{example}

%\subsection{Relationship between different definitions of
%hyperbolicity on the semi-axis $\mathbb R_{+}$}\label{Sec3.3}

Let $\mathbb T \in \{\mathbb R_{+},\mathbb R\}$, $(Y,\mathbb
R,\sigma)$ be a dynamical system on $Y$ and $\langle \mathfrak
B,\varphi, (Y,\mathbb R,\sigma)\rangle$ be a cocycle over
$(Y,\mathbb R,\sigma)$ with the fibre $\mathfrak B$.

\begin{definition}\label{defED01}
A linear cocycle $\langle \mathfrak B, \varphi, (Y,\mathbb R,
\sigma)\rangle$ is \emph{hyperbolic} (or equivalently, satisfies
the condition of \emph{exponential dichotomy}) if there exists a
continuous projection valued function $P:Y \to [\mathfrak B]$
satisfying:
\begin{enumerate}
\item $P(\sigma(t,y))U(t,y)=U(t,y)P(y)$ for any $(t,y)\in \mathbb
T\times Y$: \item there exist constants $\nu >0$ and $\mathcal
N>0$ such that
$$
\Vert U(t,y)P(y)U^{-1}(\tau,y)\Vert \le \mathcal N e^{-\nu
(t-\tau)}\ \ \forall\ t\ge \tau
$$
and
$$
\Vert U(t,y)Q(y)U^{-1}(\tau,y)\Vert \le \mathcal N e^{\nu
(t-\tau)}\ \ \forall \ t\le \tau
$$
and for any $y\in Y$, where $U(t,y):=\varphi(t,\cdot,y)$ and
$Q(y):=Id_{\mathfrak B}-P(y)$ for any $y\in Y$ and $t\in \mathbb
T$, where $Id_{\mathfrak B}$ is the identity operator on
$\mathfrak B$.
\end{enumerate}
\end{definition}

A \emph{Green's function} $G_{y}(t,\tau)$ for hyperbolic cocycle
$\varphi$ is defined by
\begin{equation}
 G_{y}(t,\tau):=\left\{\begin{array}{ll}
&\!\! U(t,y)P(y)U^{-1}(\tau,y),\;\mbox{if}\; t>\tau \ \mbox{and}\ y\in Y \\[2mm]
&\!\! -U(t,y)Q(y)U^{-1}(\tau,y), \;\mbox{if}\; t<\tau \
\mbox{and}\ y\in Y .
\end{array}
\right.\nonumber
\end{equation}

Note some properties of the Green's function that directly follow
from its definition and properties of the Cauchy operator
$U(t,y)$:
\begin{enumerate}
\item the mapping $G: Y\times (\mathbb R\setminus \{0\})\to
[\mathfrak B]$ ($G(\tau,y):=G_{y}(0,\tau)$) is continuous; \item
$G(+0,y)-G(-0,y)=Id_{\mathfrak B}$\ for any $y\in Y$, where
$Id_{\mathfrak B}$\ is the unit operator acting on he space
$\mathfrak B$; \item $G_{y}(t+h,\tau)=G_{\sigma(h,y)}(t,\tau-h)$
for any $h\in \mathbb R$ and $(t,y)\in (\mathbb R\setminus
\{0\})\times Y$.
\end{enumerate}

Let $A\in C(\mathbb T,[\mathfrak B])$, $\Sigma_{A}:=\{A_{\tau}:\
A^{\tau}(t):=A(t+\tau)$ for any $t\in \mathbb T\}$ and
$H(A):=\overline{\Sigma}_{A}$\index{$H(A)$}, where by bar is
denoted the closure of set $\Sigma_{A}$ in $C(\mathbb T,[\mathfrak
B])$. If $\mathbb T =\mathbb R_{+}$, then we denote by $H^{+}(A)$
the set $H(A)$.

Consider a linear differential equation
\begin{equation}\label{eqUC3}
x'=A(t)x \ \ (A\in C(\mathbb R,[\mathfrak B])).
\end{equation}
Along with the equation (\ref{eqUC3}) we consider its
$H^{+}$-class, that is, the family of the equations
\begin{equation}\label{eqUC4}
x'=B(t)x \ \ (B\in H^{+}(A)).\nonumber
\end{equation}

\begin{definition}\label{def3.2*}
Recall \cite[Ch.I]{DK_1970} that a linear bounded operator $P:
\mathfrak B\to \mathfrak B$ is called a projection, if $P^2=P,$
where $P^2:=P\circ P$.
\end{definition}

\begin{definition}\label{defED1}
Let $U(t,A)$ be the operator of Cauchy (a solution operator) of
linear equation $(\ref{eqUC3})$. Following \cite[Ch.IV]{DK_1970}
we will say that equation (\ref{eqUC3}) has an exponential
dichotomy (is hyperbolic) on $I\subseteq \mathbb R$, if there
exists a projection $P(A)\in [\mathfrak B]$ and constants
$\mathcal N\ge 1, \nu >0$ such that the following conditions hold:
\begin{enumerate}
\item
\begin{equation}\label{eqED1} \Vert
U(t,A)P(A)U^{-1}(\tau,A)\Vert \le \mathcal N e ^{-\nu (t-\tau)}
(\forall t> \tau;\ t,\tau \in I)\nonumber
\end{equation}
and
\item
\begin{equation}\label{eqED2}
\Vert U(t,A)Q(A)U^{-1}(\tau,A)\Vert \le \mathcal Ne^{\nu (t-\tau)}
(\forall t < \tau :\ t,\tau \in I),\nonumber
\end{equation}
where $Q(A):=Id_{\mathfrak B}-P(A)$.
\end{enumerate}
\end{definition}

\begin{remark}\label{remDE1} Assume that $P\in [\mathfrak B]$ is a
projector and $U(t,A)$ is the Cauchy operator of the equation
$(\ref{eqUC3})$ then:
\begin{enumerate}
\item the operator $P(t):=U(t,A)PU^{-1}(t,A)\in [\mathfrak B]$ is
a projector for any $t\in I$; \item $U(t,A)P=P(t)U(t,A)$ for any
$t\in I$.
\end{enumerate}
\end{remark}

Along with classical definition of hyperbolicity (Definition
\ref{defED1}) we will use an other definition given below. And in
this Subsection we establish the relation between two definitions
of hyperbolicity for linear homogeneous differential equations
with continuous (bounded) coefficients.

\begin{definition}\label{defED2} A differential equation
(\ref{eqUC3}) is said to be \emph{hyperbolic (satisfies the
condition of exponential dichotomy)} if there are two projections:
$P,Q:H^{+}(A)\mapsto [\mathfrak B]$ ($P^2(B)=P(B)$ and
$Q^2(B)=Q(B)$ for any $B\in H^{+}(A)$) such that
\begin{enumerate}
\item the mappings $P$ and $Q$ are continuous; \item
$P(B)+Q(B)=Id_{\mathfrak B}$ for any $B\in H^{+}(A)$; \item
$U(t,B)P(B)=P(B^{t})U(t,B)$ for any $t\in\mathbb R_{+}$ and $B\in
H^{+}(A)$, where $U(t,B):=\varphi(t,\cdot,B)$; \item there are
numbers $\mathcal N\ge 1$ and $\nu >0$ such that
$||U(t,B)P(B)||\le \mathcal Ne^{-\nu t}$ for any $t\ge 0$ and
$||U(t,B)Q(B)||\le \mathcal Ne^{\nu t}$ for any $t\le 0$.
\end{enumerate}
\end{definition}

\begin{remark}\label{remUC2} Note that the definition above of
hyperbolicity on $\mathbb R_{+}$ means that the cocycle $\langle
\mathfrak B,\varphi,(H^{+}(A),\mathbb R_{+},\sigma)\rangle$
generated by the equation (\ref{eqUC3}) is hyperbolic in the sense
of Definition \ref{defED01}.
\end{remark}

\begin{remark}\label{remDP1} 1. It is clear that from Definition
\ref{defED2} it follows Definition \ref{defED1}.

2. Below we will show that for the equation (\ref{eqUC3}) with
bounded coefficients (i.e., there exists a positive constant $a$
such that $||A(t)||\le a$ for any $t\in \mathbb R$) under
Condition (\textbf{C}) Definitions \ref{defED2} and \ref{defED1}
are equivalent.
\end{remark}

\begin{lemma}\label{lDP01} \cite{Che_2024.2} Let $A\in C(\mathbb R_{+},[\mathfrak B])$.
If the equation (\ref{eqUC3}) is hyperbolic in the sense of
Definition \ref{defED2}, then it is also so in the classical
sense.
\end{lemma}

Condition (\textbf{C}). The family of projectors
$\{P(t)\}=\{U(t,A)PU^{-1}(t,A)\}_{t\in \mathbb R}$ is a precompact
set of the Banach space $[\mathfrak B]$.

\begin{remark}\label{remED2} If the projector $P$ is finite-dimensional
(i.e., $P(\mathfrak B)$ is a finite-dimensional subspace of the
Banach space $\mathfrak B$) and there exists a positive constant
$C>0$ such that $\|U(t,A)PU^{-1}(t,A)\|\le C$ for any $t\in
\mathbb R$, then the family of operators
$\{U(t,A)PU^{-1}(t,A)\}_{t\in \mathbb R}$ is a precompact subset
of $[\mathfrak B]$.
\end{remark}

\begin{lemma}\label{lDP2} \cite{Che_2024.2} Assume that Condition (\textbf{C}) holds.
Then if the equation (\ref{eqUC3}) is hyperbolic (in the sense of
Definition \ref{defED1}), then it is also hyperbolic in the sense
of Definition \ref{defED2}.
\end{lemma}

\section{RAP Solutions of Linear Differential
Equations}\label{Sec4}

Consider a linear differential equation
\begin{equation}\label{eqLE1}
x'=A(t)x,
\end{equation}
where $A\in C(\mathbb R,[\mathfrak B])$. Along with the equation
(\ref{eqLE1}) consider the family of equations
\begin{equation}\label{eqLE2}
x'=B(t)x,\nonumber
\end{equation}
where $B\in H^{+}(A)$.

\begin{definition} A function $g\in C(\mathbb{R},\mathfrak B)$
is called $\omega$(respectively, $\alpha)$--limit for $f\in
C(\mathbb R,\mathfrak B)$, if there exists a sequence
$t_n\to+\infty$ (respectively, $-\infty)$ such that $f^{t_n}\to g$
in the topology of the space $C(\mathbb{R},\mathfrak B)$.
\end{definition}

By $\omega_f$ (respectively, $\alpha_f $) there is denoted the set
of all $\omega$ (respectively, $\alpha$)--limit functions for $f$
and $\Delta_f:=\omega_f\bigcup
\alpha_f$\index{$\Delta_f$,$\omega_f$,$\alpha_f$}. Assume that
\begin{equation*}
  H^{+}(f):=\overline{\{f^{\tau}:\ \tau\in\mathbb{R}^{+}\}},
  \ \ H^{-}(f):=\overline{\{f^{\tau}:\ \tau\in\mathbb{R}^{-}\}}
\end{equation*}
and
\begin{equation*}
  H(f):=H^{+}(f)\cup H^{-}(f),\index{$H(f)$, $H^{+}(f)$, $H^{-}(f)$}
\end{equation*}
where by bar there is denoted the closure in the topology of the
space $C(\mathbb{R},\mathfrak B)$.

Consider a differential equation
\begin{equation}\label{eqED3}
  \frac{dx}{dt}=A(t)x,
\end{equation}
where $A\in C(\mathbb{R},[\mathfrak B])$. Denote by $U(t,A)$ the
Cauchy operator of the equation (\ref{eqED3}) and
\begin{equation*}
  \varphi(t,A,x)=U(t,A)x.
\end{equation*}

\begin{lemma}\label{lED1} \cite{Che_2024.2} Under the Condition (C) if
the equation (\ref{eqLE1}) is hyperbolic on the semi-axis $\mathbb
R_{+}$ (respectively, on the semi-axis $\mathbb R_{-}$), then for
any $B\in \omega_{A}$ (respectively, $B\in \alpha_{A}$) the
equation
\begin{equation}\label{eqED4}
  \dot y=B(t)y
\end{equation}
is hyperbolic on $\mathbb{R}$.
\end{lemma}

\begin{coro}\label{cor*4.2.2}
Assume that the equation (\ref{eqED3}) is hyperbolic on
$\mathbb{R}_+$ (respectively, on the semi-axis $\mathbb R_{-}$),
the Condition (\textbf{C}) is fulfilled and $B\in\omega_A$
(respectively, $B\in \alpha_{A}$). Then the equation (\ref{eqED4})
has not nonzero bounded on $\mathbb{R}$ solutions.
\end{coro}

\begin{coro}\label{cor*4.2.3}
Assume tat the equation (\ref{eqED3}) is hyperbolic on $\mathbb
R_{+}$ (respectively, on the semi-axis $\mathbb R_{-}$), Condition
(\textbf{C}) holds and $A$ is positively Poisson stable, i.e.,
$A\in\omega_A$ (respectively, $A$ is negatively Poisson stable,
i.e., $A\in \alpha_{A}$). Then the equation (\ref{eqED3}) is
hyperbolic on $\mathbb{R}$.
\end{coro}

\begin{definition}\label{defCS1} Let $\mathbb T\in \{\mathbb R_{+},\mathbb R\}$.
A solution $\varphi \in C(\mathbb T,\mathfrak
B)$ of the equation
\begin{equation}\label{eqLEf1}
x'=A(t)x+f(t)
\end{equation}
is said to be positively remotely compatible if $\mathfrak
L_{(A,f)}^{+\infty}\subseteq \mathfrak L_{\varphi}^{+\infty}$.
\end{definition}

\begin{remark}\label{remCS1} Assume that $(A,f)\in C(\mathbb R_{+},[\mathfrak B])\times C(\mathbb R_{+},\mathfrak
B)$ is remotely stationary (respectively, remotely
$\tau$-periodic, remotely almost periodic). If $\varphi \in
C(\mathbb R_{+},\mathfrak B)$ is a compatible solution of the
equation (\ref{eqLEf1}), then the solution $\varphi$ is also
remotely stationary (respectively, remotely $\tau$-periodic,
remotely almost periodic).
\end{remark}

This statement follows directly from Definition \ref{defCS1} and
Theorem \ref{thRAP4}.

\begin{theorem}\label{thRAPDE1} \cite{Che_2024.2} Assume that $(A,f)\in C(\mathbb
R,[\mathfrak B])\times C(\mathbb R,\mathfrak B)$ and the following
conditions are fulfilled:
\begin{enumerate}
\item the functions $A$ and $f$ are positively Lagrange stable;
\item the equation (\ref{eqLE1}) is hyperbolic on the semi-axis
$\mathbb R_{+}$ with the projection $P(A)$.
\end{enumerate}

Then
\begin{enumerate}
\item the equation (\ref{eqLEf1}) admits at least one remotely
compatible solution $\varphi$; \item
\begin{equation}\label{eqF1}
\varphi(t)=\int_{0}^{+\infty}G_{A}(t,\tau)f(\tau)d\tau,
\end{equation}
where
\begin{equation}\label{eqF2}
 G_{A}(t,\tau):=\left\{\begin{array}{ll}
&\!\! U(t,A)P(A)U^{-1}(\tau,A),\;\mbox{if}\; t>\tau \\[2mm]
&\!\! -U(t,A)(I-P(A))U^{-1}(\tau,A), \;\mbox{if}\; t<\tau ,
\end{array}
\right.\nonumber
\end{equation}
and $U(t,A)$ is the Cauchy operator of the equation (\ref{eqLE1});
\item the Green's operator $\mathbb G_{A}$ defined by
\begin{equation}\label{eqG1}
\mathbb G_{A}(f)(t):=\int_{0}^{+\infty}G_{A}(t,\tau)f(\tau)d\tau \
\ (t\in \mathbb R_{+}) \nonumber
\end{equation}
is a linear bounded operator defined on $C_{b}(\mathbb
R_{+},\mathfrak B)$ with value in $C_{b}(\mathbb R_{+},\mathfrak
B)$ and
\begin{equation}\label{eqG2}
\|\mathbb G_{A}(f)\|\le \frac{2\mathcal N}{\nu}\|f\|
\end{equation}
for any $f\in C_{b}(\mathbb R_{+},\mathfrak B)$.
\end{enumerate}
\end{theorem}

\begin{coro}\label{corR_1} \cite{Che_2024.2} Under the conditions of Theorem
\ref{thRAPDE1} every bounded on $\mathbb R_{+}$ solution $\psi$ of
the equation (\ref{eqLEf1}) is remotely compatible, i.e.,
$\mathfrak L_{(A,f)}^{+\infty}\subseteq \mathfrak
L_{\psi}^{+\infty}$.
\end{coro}

\begin{coro}\label{corR1}
Under the conditions of Theorem \ref{thRAPDE1} if the coefficients
$(A,f)$ are remotely stationary (respectively, remotely
$\tau$-periodic, remotely almost periodic), then
\begin{enumerate}
\item the equation (\ref{eqLEf1}) admits at least one remotely
stationary (respectively, remotely $\tau$-periodic, remotely
almost periodic) solution $\varphi$; \item
\begin{equation}\label{eqF1_1}
\varphi(t)=\int_{0}^{+\infty}G_{A}(t,\tau)f(\tau)d\tau ;\nonumber
\end{equation}
\item every bounded on $\mathbb R_{+}$ solution $\psi$ of the
equation (\ref{eqLEf1}) is remotely stationary (respectively,
remotely $\tau$-periodic, remotely almost periodic).
\end{enumerate}
\end{coro}
\begin{proof} This statement follows from Theorem \ref{thRAPDE1},
Corollary \ref{corR_1} and Remark \ref{remCS1}.
\end{proof}

\begin{remark}\label{remRAP_01} Note that all results from Sections
\ref{Sec3} and \ref{Sec4} take place also on the negative
semi-axis $\mathbb R_{-}$. This fact directly follows from the
results established in Sections \ref{Sec3} and \ref{Sec4} by
replacing $t$ with $-t$.
\end{remark}

\section{Two-sided RAP Solutions of Linear Differential
Equations}\label{Sec5}

%\subsubsection{Linear Differential Equations with Exponential
%Trichotomy}\label{Sec13.1.1}

Consider a linear differential equation
\begin{equation}\label{eqET1}
x'=A(t)x,
\end{equation}
where $A\in C(\mathbb R,[\mathfrak B])$, and denote by $U(t,A)$
the Cauchy operator of the equation (\ref{eqET1}).

\begin{definition}\label{defET1} \cite{EH_1988} A linear differential equation
(\ref{eqET1}) is said to have an exponential trichotomy (on
$\mathbb R$) if there exist linear projections $P, Q$ such that
\begin{equation}\label{eqET2}
PQ=QP, P+Q-PQ=I, \nonumber
\end{equation}
and constants $K\ge 1$ and $\alpha >0$ such that
\begin{eqnarray}\label{eqET3}
& \|U(t,A)PU^{-1}(\tau,A)\|\le \mathcal N e^{-\nu (t-\tau)}\ \
\mbox{for}\ \ 0\le \tau \le t,\nonumber \\
& \|U(t,A)(I- P)U^{-1}(\tau,A)\|\le \mathcal N e^{-\nu (\tau -t)}\
\
\mbox{for}\ \ t\le \tau, \tau \ge 0 \nonumber \\
& \|U(t,A)QU^{-1}(\tau,A)\|\le \mathcal N e^{-\nu (\tau-t)}\ \
\mbox{for}\ \
t\le \tau \le 0 \\
& \|U(t,A)(I-Q)U^{-1}(\tau,A)\|\le \mathcal N e^{-\nu (t-\tau)}\ \
\mbox{for}\ \  \tau \le t,\ \tau \le 0.\nonumber
\end{eqnarray}
\end{definition}

\begin{remark}\label{remET1} 1. If the equation (\ref{eqET1}) has an exponential trichotomy on
$\mathbb R$, then the first two inequalities in (\ref{eqET3})
imply exponential dichotomy of (\ref{eqET1}) on $\mathbb R_{+}$
with the projection $P_{+}=P$ and the last two inequalities of
(\ref{eqET3}) imply exponential dichotomy of (\ref{eqET1}) on
$\mathbb R_{-}$ with the projection $P_{-}=I-Q$.

2. If we take $Q=I-P$, then the equation (\ref{eqET1}) is
exponentially dichotomic on $\mathbb R$.
\end{remark}

\begin{definition}\label{defGF1} If the equation (\ref{eqET1}) has the exponential
trichotomy on $\mathbb R$, then for (\ref{eqET1}) can be
constructed a Green function \cite{EH_1988,PPX_2023} by the
equality
\begin{equation}\label{eqET3.1}
G(t,\tau)=\left\{ \begin{array}{ll}
\displaystyle{U(t,A)PU^{-1}(\tau,A)}, \ \ \ (0\le \tau < t)\\
\\
\displaystyle{-U(t,A)(I-P)U^{-1}(\tau,A)}, \ \ \ (t\le \tau,\ \tau \ge 0)\\
\\
\displaystyle{-U(t,A)QU^{-1}(\tau,A)}, \ \ \ (t\le \tau \le 0)\\
\\
\displaystyle{U(t,A)(I-Q)U^{-1}(\tau,A), \ \ \ (\tau \le t,\ \tau
\le 0)}.
\end{array}\right.
\end{equation}
\end{definition}

\begin{lemma}\label{lET_1} \cite[Ch.III,p.145]{Che_2015}, \cite{EH_1988,PPX_2023} Assume that the equation (\ref{eqET1})
has an exponential trichotomy on $\mathbb R$. Then
\begin{enumerate}
\item for any $f\in C_{b}(\mathbb R,\mathfrak B)$ the
nonhomogeneous equation
\begin{equation}\label{eqET1n}
x'=A(t)x+f(t)\nonumber
\end{equation}
has a bounded on $\mathbb R$ solution $\varphi$; \item the
solution $\varphi$ is defined by the equality
\begin{equation}\label{eqET3.2}
\varphi(t)=\int_{-\infty}^{+\infty}G(t,\tau)f(\tau)d\tau,\nonumber
\end{equation}
where $G(t,\tau)$ is the Green function of the equation
(\ref{eqET1}) defined by (\ref{eqET3.1}); \item $R\varphi(0)=0$,
where $R:=PQ$; \item $\|\varphi\|_{b}\le \frac{2\mathcal
N}{\nu}\|f\|_{b},$ where $\|\varphi\|_{b}:=\sup\{|\varphi(t)|\
t\in \mathbb R\}$.
\end{enumerate}
\end{lemma}

\begin{lemma}\label{lET1} \cite{EH_1988,PV_2011} The following
statements are pairwise equivalent:
\begin{enumerate}
\item the equation (\ref{eqET1}) has an exponentially trichotomy
on $\mathbb R$; \item the equation (\ref{eqET1}) has an
exponential dichotomy on $\mathbb R_{\pm}$ (that is, exponential
dichotomy on $\mathbb R_{+}$ and on $\mathbb R_{-}$) with
projections $P_{+}$, and $P_{-}$, respectively, such that
$P_{+}P_{-} = P_{-} P_{+} = P_{-}$; \item there are three mutually
orthogonal projections $P_{1}, P_{2}, P_{3}$, with sum $I$ and
such that
\begin{eqnarray}\label{eqET4}
& \|U(t,A)P_{1}U^{-1}(\tau,A)\|\le \mathcal N e^{-\nu (t-\tau)} \
\
\mbox{for} \ \ t\ge \tau, \nonumber \\
&
 \|U(t,A)P_{2}U^{-1}(\tau,A)\|\le \mathcal N e^{-\nu (\tau-t)} \ \
\mbox{for} \ \ t\le \tau, \nonumber \\
&
 \|U(t,A)P_{3}U^{-1}(\tau,A)\|\le \mathcal N e^{-\nu (t-\tau)} \ \
\mbox{for} \ \ t\ge \tau \ge 0, \nonumber \\
&
 \|U(t,A)P_{3}U^{-1}(\tau,A)\|\le \mathcal N e^{-\nu (\tau-t)} \ \
\mbox{for} \ \ t\le \tau \le 0 .\nonumber
\end{eqnarray}
\end{enumerate}
\end{lemma}

\begin{definition}\label{defTS1} A function $\varphi \in C(\mathbb R,\mathfrak
B)$ is said to be two-sided remotely periodic if it is positively
and negatively remotely $\tau$-periodic.
\end{definition}

%Let $\varphi \in C(\mathbb R,\mathfrak B)$. Denote by $\mathfrak
%L^{\pm \infty}_{\varphi}:=\{\{t_n\}\subset \mathbb R|$ such that
%the sequence $\{\varphi^{t_n}\}$ converges in the space $C(\mathbb
%R,\mathfrak B)$ and $t_n\to \pm \infty$ as $n\to \infty\}$ and
%$\mathfrak L_{\varphi}:=\mathfrak L^{-\infty}_{\varphi}\bigcup
%\mathfrak L^{+\infty}_{\varphi}$.

\begin{theorem}\label{thTSRAP_1} Assume that $(A,f)\in C(\mathbb R,[\mathfrak B])\times C(\mathbb R,\mathfrak
B)$ and the following conditions are fulfilled:
\begin{enumerate}
\item $(A,f)$ is Lagrange stable; \item the equation (\ref{eqLE1})
has an exponential trichotomy on $\mathbb R$.
\end{enumerate}

Then
\begin{enumerate}
\item the equation (\ref{eqLEf1}) has at least one two-sided
remotely compatible solution $\varphi$, i.e., $\mathfrak
L_{(A,f)}\subseteq \mathfrak L_{\varphi}$; \item the solution
$\varphi$ is defined by
\begin{equation}\label{eqTS1}
\varphi(t)=\int_{-\infty}^{+\infty}G_{A}(t,\tau)f(\tau)d\tau,
\end{equation}
where $G_{A}(t,\tau)$ is the Green function of the equation
(\ref{eqET1}) defined by (\ref{eqET3.1}); \item every bounded on
$\mathbb R$ solution of the equation (\ref{eqLEf1}) is two-sided
remotely compatible.
\end{enumerate}
\end{theorem}
\begin{proof} Since the function $f$ (respectively, the operator-function $A$)
is Lagrange stable, then it is bounded on $\mathbb R$. Assume that
the equation (\ref{eqET1}) has an exponential trichotomy on
$\mathbb R$ then by Lemma \ref{lET_1} the equation (\ref{eqLEf1})
has at least one bounded on $\mathbb R$ solution $\varphi$ which
is defined by (\ref{eqTS1}). Taking into account that under the
condition of Theorem the equation (\ref{eqET1}) has an exponential
trichotomy on $\mathbb R$ we conclude that it admits the
exponential dichotomy on $\mathbb R_{+}$ (respectively, on
$\mathbb R_{-}$). By Theorem \ref{thRAPDE1} and Corollary
\ref{corR_1} (see also Remark \ref{remRAP_01}) we have $\mathfrak
L^{\pm \infty}_{(A,f)}\subseteq \mathfrak L^{\pm
\infty}_{\varphi}$ and, consequently, $\mathfrak
L_{(A,f)}\subseteq \mathfrak L_{\varphi}$.

If $\bar{\varphi}$ is an arbitrary bounded on $\mathbb R$ solution
of the equation (\ref{eqLEf1}), then using the same arguments as
above we can prove that the solution $\bar{\varphi}$ is two-sided
remotely compatible.
\end{proof}

\begin{coro}\label{corTSRAPLD1} Under the conditions of
Theorem \ref{thTSRAP_1} if the equation (\ref{eqLE1}) satisfies
the condition of exponential trichotomy on $\mathbb R$, then
\begin{enumerate}
\item the equation (\ref{eqLEf1}) has a unique two-sided remotely
compatible solution $\varphi$; \item the solution $\varphi$ is
defined by
\begin{equation}\label{eqTS1.1}
\varphi(t)=\int_{-\infty}^{+\infty}G_{A}(t,\tau)f(\tau)d\tau
.\nonumber
\end{equation}
\end{enumerate}
\end{coro}
\begin{proof} This statement follows from Theorem \ref{thTSRAP_1}
because the equation (\ref{eqLE1}) with exponential dichotomy on
$\mathbb R$ has no nonzero bounded on $\mathbb R$ solutions.
\end{proof}

\begin{theorem}\label{thTSRAP1} Assume that $(A,f)\in C(\mathbb R,[\mathfrak B])\times C(\mathbb R,\mathfrak
B)$ and the following conditions are fulfilled:
\begin{enumerate}
\item $(A,f)$ is Lagrange stable and two-sided remotely almost
periodic (respectively, two-sided remotely $\tau$-periodic or
two-sided remotely stationary); \item the equation (\ref{eqLE1})
has an exponential trichotomy on $\mathbb R$.
\end{enumerate}

Then
\begin{enumerate}
\item the equation (\ref{eqLEf1}) has at least one two-sided
remotely almost periodic (respectively, two-sided remotely
$\tau$-periodic or two-sided remotely stationary) solution
$\varphi$; \item the solution $\varphi$ is defined by
\begin{equation}\label{eqTS_1}
\varphi(t)=\int_{-\infty}^{+\infty}G_{A}(t,\tau)f(\tau)d\tau,\nonumber
\end{equation}
where $G_{A}(t,\tau)$ is the Green function of the equation
(\ref{eqET1}) defined by (\ref{eqET3.1}); \item every bounded on
$\mathbb R$ solution of the equation (\ref{eqLEf1}) is two-sided
remotely almost periodic (respectively, two-sided remotely
$\tau$-periodic or two-sided remotely stationary).
\end{enumerate}
\end{theorem}
\begin{proof} This statement follows from Theorem \ref{thTSRAP_1}  and
Corollary \ref{corR1}.
\end{proof}

\section{Two-sided RAP Solutions of Perturbed Linear Differential
Equations}\label{Sec6}

Let $W\subseteq X$ be a bounded (respectively, compact) subset of
$X$. Denote by $C(\mathbb R\times W,X)$ the space of all
continuous functions $f :\mathbb R\times W \to X$ equipped with
the compact-open topology. On the space $C(\mathbb \mathbb R\times
W,X)$ is defined \cite[Ch.I]{Che_2015} a shift dynamical system
(Bebutov's dynamical system) $(C(\mathbb R\times W,X),\mathbb
R,\sigma)$, where $\sigma :\mathbb R\times C(\mathbb R\times
W,X)\to C(\mathbb R\times W,X)$ is defined by equality
$\sigma(\tau,f)=f^{\tau}$ and $f^{\tau}(t,x)=f(t+\tau,x)$ for any
$f\in C(\mathbb R\times W,X)$, $t,\tau \in \mathbb R$ and $x\in
W$.

If $Q$ is a compact subset of $W$, then the topology on the space
$C(\mathbb R\times Q,X)$ can be \cite{Shc_1967} defined by
distance
\begin{equation}\label{eqD1}
d(f,g):=\sup\limits_{l>0}\min\{\max\limits_{|t|\le l,\ x\in
Q}\rho(f(t,x),g(t,x)), l^{-1}\} .\nonumber
\end{equation}

\begin{definition}\label{defAA1} A function $f\in C(\mathbb
R\times W,X)$ is said to be (two-sided) remotely almost periodic
in $t\in \mathbb R$ uniformly with respect to $x\in W$ if for
arbitrary positive number $\varepsilon$ there exists a relatively
dense in $\mathbb R$ subset $\mathcal P(\varepsilon,f,W)$ such
that for any $\tau \in \mathcal P(\varepsilon,f,W)$ there exists a
number $L(\varepsilon,f,W,\tau)>0$ for which we have
\begin{equation}\label{eqAA2}
\sup\limits_{x\in W}\rho(f(t+\tau,x),f(t,x))<\varepsilon \nonumber
\end{equation}
for any $|t|\ge L(\varepsilon,f,W,\tau)$.
\end{definition}

If $Q$ is a compact subset of $W$ and $f\in C(\mathbb R\times
W,X)$ then we denote by $f_{Q}$ the restriction $f$ on $\mathbb
R\times Q$, i.e., $f_{Q}:=f\big{|}_{\mathbb R\times Q}$ and
$\mathfrak L_{f_{Q}}:=\{\{t_{n}\}\subset \mathbb R|\
\{f_{Q}^{t_n}\}$ converges in $C(\mathbb R\times Q,X)$ and $t_n\to
+\infty$ as $n\to \infty$ $\}$.

Note that $C(Q,X)$ is a complete metric space equipped with the
distance
\begin{equation}\label{eqD2}
d(F_1,F_2):=\max\limits_{x\in Q}\rho(F_1(x),F_2(x)).\nonumber
\end{equation}

Let $Q$ be a compact subset of $X$. Consider the functional spaces
$C(\mathbb R\times Q,X)$ and $C(\mathbb R,C(Q,X))$. For any $f\in
C(\mathbb R\times Q,X)$ corresponds \cite{Shc_1967} a unique map
$F\in C(\mathbb T,C(Q,X))$ defined by the equality
\begin{equation}\label{eqAA4}
F_{f}(t):=f(t,\cdot). \nonumber
\end{equation}
This means that it is well defined the mapping $h:C(\mathbb
R\times Q,X)\to C(\mathbb R,C(Q,X))$ by the equality
\begin{equation}\label{eqD_3}
h(f)=F_{f}.\nonumber
\end{equation}

The following statement holds.

\begin{theorem}\label{thAA1} \cite{Shc_1967} Let $Q$ be a compact subset of $X$,
then the mapping $h: C(\mathbb R\times Q,X)\to C(\mathbb
R,C(Q,X))$ possesses the following properties:
\begin{enumerate}
\item $h$ is a homeomorphism of the space $C(\mathbb R\times Q,X)$
onto $C(\mathbb R,C(Q,X))$; \item for any $l>0$ and $f,g\in
C(\mathbb R\times Q,X)$ we have
\begin{equation}\label{eqH_1}
\max\limits_{|t|\le l,\ x\in
Q}\rho(f(t,x),g(t,x))=\max\limits_{|t|\le
l}\rho(F_{f}(t),F_{g}(t));\nonumber
\end{equation}
\item $d(f,g)=d(F_{f},F_{g})$ for any $f,g\in C(\mathbb R\times
Q,X)$,
\end{enumerate}
i.e., the mapping $h$ defines an isometry between $C(\mathbb
R\times Q,X)$ and $C(\mathbb R,C(Q,X))$.
\end{theorem}

\begin{lemma}\label{lAA0} \cite{Che_2024.2} The following relation
\begin{equation}\label{eqH}
h(\sigma(\tau,f))=\sigma(\tau,h(f))\nonumber
\end{equation}
holds for any $(\tau,f)\in \mathbb R\times C(\mathbb R\times
Q,X)$, i.e., $h$ is an isometric homeomorphism of the dynamical
system $(C(\mathbb R\times Q,X),\mathbb R,\sigma)$ onto
$(C(\mathbb R,C(Q,X)),\mathbb R,\sigma)$.
\end{lemma}

\begin{definition}\label{defLS2} A function $F\in C(\mathbb R\times
Q,X)$ is said to be Lagrange stable if the set
$\Sigma_{F}:=\{\sigma(\tau,F)|\ \tau\in \mathbb R\}$ is
pre-compact in $C(\mathbb R\times Q,X)$.
\end{definition}

\begin{lemma}\label{lLS2} \cite[Ch.IV]{bro75}, \cite[Ch.III]{Sch72}, \cite{SK_1974}
Assume that $Q$ is a compact subset of $X$. A function $F\in
C(\mathbb R\times Q,X)$ is Lagrange stable if and only if the
following conditions hold:
\begin{enumerate}
\item the set $F(\mathbb R\times Q)$ is pre-compact in $X$; \item
the mapping $F:\mathbb R\times Q\to X$ is uniformly continuous.
\end{enumerate}
\end{lemma}

\begin{lemma}\label{lLS3} \cite[Ch.III]{Sch72}, \cite{SK_1974}  Let $\varphi\in C(\mathbb
R,X)$be a Lagrange stable function and
$Q:=\overline{\varphi(\mathbb R)}$, where by bar is denoted the
closure in $X$. If the function $F\in C(\mathbb R\times Q,X)$ is
Lagrange stable, then the function $\psi \in C(\mathbb R,X)$
defined by equality $\psi(t):=F(t,\varphi(t))$ for any
$t\in\mathbb R$ is also Lagrange stable.
\end{lemma}

\begin{theorem}\label{thAA2} Let $Q$ be a compact subset of $X$.
Then the following statements are equivalent: \begin{enumerate}
\item the function $f\in C(\mathbb R\times Q,X)$ is remotely
almost periodic (respectively, remotely $\tau$-periodic or
remotely stationary) in $t\in \mathbb R$ uniformly w.r.t. $x\in
Q$; \item the function $F:=h(f)\in C(\mathbb T,C(Q,X))$ is
remotely almost periodic (respectively, remotely $\tau$-periodic
or remotely stationary).
\end{enumerate}
\end{theorem}
\begin{proof} This
statement follows directly from the corresponding definitions,
Theorem \ref{thAA1} and Lemma \ref{lAA0}.
\end{proof}

In this section we establish the conditions, under which the
existence of a compatible in the limit solution of a semi-linear
equation with hyperbolic linear part.

Denote by $\Delta$ the family of all increasing (nondecreasing)
functions $L :\mathbb R_{+}\to\mathbb R_{+}$ with
$\lim\limits_{r\to 0^{+}}L(r)=0$.

\begin{definition}\label{defLE1} A function $F\in C(\mathbb T\times\mathfrak B,\mathfrak
B)$ is said to be
\begin{enumerate}
\item local Lipschitzian with respect to variable $u\in \mathfrak
B$ on $\mathfrak B$ uniformly with respect to $t\in \mathbb T$, if
there exists a function $L:\mathbb R_{+}\to \mathbb R_{+}$ such
that
\begin{equation}\label{eqLE_01}
|F(t,u_1)-F(t,u_2)|\le L(r)|u_1-u_2|
\end{equation}
for all $u_1,u_2\in B[0,r]$ and $t\in\mathbb R$, where
$B[0,r]:=\{u\in \mathfrak B:\ |u|\le r\}$; \item (global)
Lipschitzian with respect to variable $u\in \mathfrak B$ on
$\mathfrak B$ uniformly with respect to $t\in \mathbb R$, if there
exists a positive constant $L$ such that
\begin{equation}\label{eqLE1.1}
|F(t,u_1)-F(t,u_2)|\le L|u_1-u_2|
\end{equation}
for all $u_1,u_2\in \mathfrak B$ and $t\in\mathbb R$.
\end{enumerate}
\end{definition}

\begin{definition}\label{defLE2} The smallest constant figuring in
(\ref{eqLE_01}) (respectively, in (\ref{eqLE1.1})) is called
Lipschitz constant of function $F$ on $\mathbb R\times B[0,r]$
(respectively, on $\mathbb R\times \mathfrak B$). Notation
$Lip(r,F)$ (respectively, $Lip(F)$).
\end{definition}

\begin{remark}\label{remL1} Note the function $Lip(\cdot,F):\mathbb R_{+}\to \mathbb
R_{+}$ is nondecreasing, that is, for all $r_1,r_2\in \mathbb
R_{+}$ such that $r_1\le r_2$ one has $Lip(r_1,F)\le Lip(r_2,F)$.
\end{remark}

\begin{lemma}\label{lLC1} Assume that
$F\in C(\mathbb R\times \mathfrak B,\mathfrak B)$ and it is local
(respectively, global) Lipschitzian with respect to variable $u\in
\mathfrak B$ on $\mathfrak B$ uniformly with respect to $t\in
\mathbb R$, i.e., there exists a function $L:\mathbb R_{+}\to
\mathbb R_{+}$ (respectively, a positive constant $L$) such that
(\ref{eqLE_01}) holds (respectively, (\ref{eqLE1.1})) for all
$u_1,u_2\in \mathfrak B$ and $t\in\mathbb R$.

Then every function $G\in H(F):=\overline{\{\sigma(\tau,F)|\ \tau
\in \mathbb R\}}$ is local (respectively, global) Lipschitzian
with respect to variable $u\in \mathfrak B$ on $\mathfrak B$
uniformly with respect to $t\in \mathbb R$ with the same function
$L(\cdot,F):\mathbb R_{+}\to \mathbb R_{+}$ (respectively, the
same positive constant $L$).
\end{lemma}
\begin{proof} This statement is evident.
\end{proof}

Consider a differential equation
\begin{equation}\label{eqSL1}
u'=A(t)u + F(t,u)
\end{equation}
in the Banach space $\mathfrak B$, where $A\in C(\mathbb
R,\mathfrak B)$ and $F$ is a nonlinear continuous mapping ("small"
perturbation) acting from $\mathbb R\times \mathfrak B$ into
$\mathfrak B$.

Along with the equation (\ref{eqSL1}) we consider irs $H$-class,
i.e., the family of equations
\begin{equation}\label{eqSL1.001}
v'=B(t)v+G(t,v)\ \ ((B,G)\in H(A,F)),
\end{equation}
where $H(A,F):=\overline{\{(A^{\tau},F^{\tau})|\ \tau \in \mathbb
R\}}$ and by bar the closure in the product space $C(\mathbb
R,[\mathfrak B])\times C(\mathbb R\times \mathfrak B,\mathfrak B)$
is denoted.

\begin{definition}\label{defSL01}
A function $F\in C(\mathbb R\times \mathfrak B,\mathfrak B)$ is
said to be regular, if for any $v\in \mathfrak B$ and $(B,G)\in
H(A,F)$ there exists a unique solution $\varphi(t,v,B,F)$ of the
equation (\ref{eqSL1.001}) passing through the point $v$ at the
initial moment $t=0$ and defined on $\mathbb R_{+}$.
\end{definition}

Below everywhere in this Section we suppose that the function
$F\in C(\mathbb R\times \mathfrak B,\mathfrak B)$ is regular.

\begin{remark}\label{remRF1} If the function $F\in C(\mathbb R\times \mathfrak B,\mathfrak
B)$ is global Lipschitzian and $\sup\limits_{t\in \mathbb
R}|F(t,0)|<+\infty$, then it is regular (see Theorem 6.6.9 from
\cite[Ch.VI]{Che_2020}).
\end{remark}

Let $\mathfrak{L}$ some family of sequences $\{t_k\}\to \infty$
(i.e., $t_k\to +\infty$ or $-\infty$) as $k\to \infty$ and $r>0$.
Denote by
\begin{enumerate}
\item[-] $C_{r}(\mathfrak{L}):= \{\varphi : \varphi\in
C_{b}(\mathbb{R},\mathfrak B),\ \mathfrak{L}\subseteq
\mathfrak{L}_{\varphi}\ \mbox{and}\ ||\varphi||\leq r \}$;
\item[-] $BC(\mathfrak L):=\{\varphi \in C_{b}(\mathbb R,\mathfrak
B):\ \mathfrak L\subseteq \mathfrak L_{\varphi}\}$; \item[-]
$\mathcal L (\mathfrak L):=\{\varphi \in \mathcal L(\mathbb
R,\mathfrak B):\ \mathfrak L \subseteq \mathfrak L_{\varphi}\}$;
\item[-] $\mathcal L_{r}(\mathfrak L):=\{\varphi \in \mathcal L
(\mathfrak L):\ \mbox{and}\ ||\varphi||\leq r \}$.
\end{enumerate}

\begin{lemma}\label{l3.4.1} \cite[Ch.III]{Che_2009}, \cite{Sch72}
$C_{r}(\mathfrak{L})$ (respectively, $C_{b}(\mathfrak L)$ or
$\mathcal L(\mathfrak L)$) is a closed subspace of the metric
space $C_b(\mathbb{R},\mathfrak B)$.
\end{lemma}

\begin{lemma}\label{l3.4.2} \cite{Che_2024.2} Let $\varphi\in \mathcal L(\mathbb{R},\mathfrak B)$
(respectively, $\varphi \in C_{b}(\mathbb R,\mathfrak B)$), $F\in
\mathcal L(\mathbb{R}\times \mathfrak B,\mathfrak B)$
(respectively, $F\in C_{b}(\mathbb R\times \mathfrak B,\mathfrak
B)$), $g(t):=F(t,\varphi(t))$ for all $t\in\mathbb{R}$ and
$F_Q:=F|_{\mathbb{R}\times Q}$, where
$Q:=\overline{\varphi(\mathbb{R})}$. If $F_Q$ satisfies the
condition of Lipschitz with respect to the second variable with
the constant $L>0$, then
\begin{enumerate}
\item $\mathfrak L_{(F_{Q},\varphi)}\subseteq \mathfrak L_{g}$;
\item if $\{t_n\}\in \mathfrak L_{(F_{Q},\varphi)}$ and
$$
(\tilde{\varphi},\tilde{F})=\lim\limits_{n\to
\infty}(\varphi^{t_n},F_{Q}^{t_n}),
$$
then $\lim\limits_{n\to \infty}g^{t_n}=\tilde{g}$, where
$\tilde{g}(t):=\tilde{F}(t,\tilde{\varphi}(t))$ for any $t\in
\mathbb R$; \item $g\in \mathcal L(\mathbb R,\mathfrak B)$
(respectively, $g\in C_{b}(\mathbb R,\mathfrak B)$).
\end{enumerate}
\end{lemma}

Let us consider a differential equation
\begin{equation}\label{eqSL01}
u'=A(t)u + f(t) +F(t,u),
\end{equation}
where $F\in C(\mathbb{R}\times \mathfrak B,\mathfrak B)$ and
$F(t,0)=0$ for any $t\in \mathbb R$.

We put $\mathfrak B_{-}:=P(\mathfrak B)$, where $P$ is a
projection figuring in the Definition \ref{defED1}.

Assume that $(A,f)\in C(\mathbb R,[\mathfrak B])\times C(\mathbb
R,\mathfrak B)$ is Lagrange stable and $\varphi_{0}$ is a remotely
compatible solution of the equation
\begin{equation}\label{eqSL1.1}
u'=A(t)u + f(t)
\end{equation}
defined by equality (\ref{eqF1}). Denote by
$Q:=\overline{\varphi_{0}(\mathbb{R})}$ the compact subset of
$\mathfrak B$ and by $Q_{r}:=\{x\in \mathfrak B|\ \rho(x,Q)\le
r\}$, where $\rho(x,Q):=\inf\{|x-q|:\ q\in Q\}$ is the
neighborhood of the set $Q\subset \mathfrak B$ of radius $r>0$.

\begin{definition}\label{defB1.0} A function $F\in C(\mathbb R\times \mathfrak B,\mathfrak
B)$ is said to be bounded (respectively, compact) if for any
bounded subset $B$ (respectively, compact subset $Q$) from
$\mathfrak B$ the set $F(\mathbb R\times B):=\{F(t,x)|\ (t,x)\in
\mathbb R\times B\}$ (respectively, $F(\mathbb R\times Q)$) is
bounded (respectively, pre-compact) in $\mathfrak B$.
\end{definition}

Let $\varphi\in C(\mathbb R,\mathfrak B)$ be a compact solution of
the equation (\ref{eqSL01}).

\begin{definition}\label{defSL1} A solution $\varphi$ of the equation (\ref{eqSL01}) is said to
be remotely compatible by the character of recurrence (shortly
remotely compatible) if $\mathfrak L_{(f,F_{Q})}\subseteq
\mathfrak L_{\varphi}$, where $Q:=\overline{\varphi(\mathbb R)}$
and $F_{Q}:=F\big{|}_{\mathbb R\times Q}$.
\end{definition}

\begin{theorem}\label{thSL0.1} Assume that the following conditions are
fulfilled:
\begin{enumerate}
\item $\varphi$ is a compact solution of the equation
(\ref{eqSL01}); \item the functions $f$ and $F_{Q}$ remotely
stationary (respectively, remotely $\tau$-pe\-ri\-o\-dic, remotely
almost periodic); \item $\varphi$ is a remotely compatible
solution of the equation (\ref{eqSL01}).
\end{enumerate}

Then the solution $\varphi$ is also remotely stationary
(respectively, remotely $\tau$-periodic, remotely almost
periodic).
\end{theorem}
\begin{proof} This statement directly follows from the Definition
\ref{defSL1} and Theorem \ref{thRAPDE1}.
\end{proof}

\begin{theorem}\label{th4A}
Let $f\in C(\mathbb{R},\mathfrak B)$, $F\in C(\mathbb{R}\times
\mathfrak B,\mathfrak B)$ and $A\in (\mathbb R,[\mathfrak B])$.

Assume that the following conditions are fulfilled:
\begin{enumerate}
\item[$1)$] the linear equation (\ref{eqLE1}) has the exponential
trichotomy on $\mathbb R$; \item[$2)$] the functions $A,f$ and $F$
are Lagrange stable; \item[$3)$] $F$ satisfies the condition of
Lipschitz with respect to $x\in \mathfrak B$ with the constant of
Lipschitz $L<\frac{\nu}{2\mathcal N}.$
\end{enumerate}

Then
\begin{enumerate}
\item the equation $(\ref{eqSL01})$ has at least one Lagrange
stable solution $\varphi$ defined by equality
\begin{equation}\label{eqBF0}
\varphi(t)=\int_{-\infty}^{+\infty}G_{A}(t,\tau)(f(\tau)+F(\tau,\varphi(\tau)))d\tau
\nonumber
\end{equation}
and
\item
\begin{equation}\label{eqBF1}
\|\varphi -\varphi_{0}\|\le r ,
\end{equation}
where $$r:=\frac{4\mathcal N^2L\|f\|}{\nu (\nu -2\mathcal N L)}.$$
\end{enumerate}
\end{theorem}
\begin{proof}
Consider the space $\mathcal L(\mathbb R,\mathfrak B)$
(respectively, $C_{b}(\mathbb R,\mathfrak B)$). By Lemma
\ref{l3.4.1} it is a complete metric space. Define an operator
$$
\Phi :\mathcal L(\mathbb R,\mathfrak B)\to \mathcal L(\mathbb
R,\mathfrak B)
$$
(respectively
$$
\Phi : C_{b}(\mathbb R,\mathfrak B)\to C_{b}(\mathbb R,\mathfrak
B) )
$$
as follows: if $\psi\in \mathcal L(\mathbb R,\mathfrak B)$
(respectively, $f\in C_{b}(\mathbb R,\mathfrak B)$), then by Lemma
\ref{l3.4.2} $g\in \mathcal L(\mathbb R,\mathfrak B)$, where
$g(t)=F(t,\psi(t))$ for any $t\in \mathbb R$. By Theorem
\ref{thRAPDE1} the equation
\begin{equation}\label{eqD01}
\frac{dz}{dt}=A(t)z+F(t,\psi(t)+\varphi_0(t)) \nonumber
\end{equation}
has a remotely compatible solution $\gamma\in \mathcal
L(\mathbb{R},\mathfrak B)$ (respectively, $\gamma \in
C_{b}(\mathbb R,\mathfrak B)$) defined by equality
\begin{equation}\label{eqBF02}
\gamma(t)=\int_{-\infty}^{+\infty}G_{A}(t,\tau)F(\tau,\psi(\tau)+\varphi_0(\tau))d\tau
\nonumber
\end{equation}
So, $\gamma \in \mathcal L(\mathbb R,\mathfrak B)$ (respectively,
$\gamma \in C_{b}(\mathbb R,\mathfrak B)$). Let
$\Phi(\psi):=\gamma$. From the said above follows that $\Phi$ is
well defined. Let us show that the operator $\Phi$ is a
contraction. Indeed, it is easy to note that the function
$\gamma:=\gamma_1-\gamma_2=\Phi(\psi_1)-\Phi(\psi_2)$ is a
solution of the equation
$$
\frac{du}{dt}=A(t)u+F(t,\psi_1(t)+\varphi_0(t))-
                    F(t,\psi_2(t)+\varphi_0(t))
$$
and
$$
\gamma_{1}(t)-\gamma_{2}(t)=\int_{-\infty}^{+\infty}G_{A}(t,\tau)[
F(\tau,\varphi_0(\tau)+\psi_{1}(\tau))-F(\tau,\varphi_0(\tau)+\psi_{2}(\tau))]d\tau
.
$$
By Theorem \ref{thRAPDE1} (item (iii) formula (\ref{eqG2})), it is
subordinated to the estimate
$$
\begin{array}{c}
||\Phi(\psi_1)-\Phi(\psi_2)||\leq\\
\frac{2\mathcal N}{\nu}\sup\limits_{t\in \mathbb R}
|F(t,\psi_1(t)+\varphi_0(t))-F(t,\psi_2(t)+\varphi_0(t))|\leq \\
\frac{2\mathcal N}{\nu}
L||\psi_1-\psi_2||=\alpha||\psi_1-\psi_2||.
\end{array}
$$
Since $\alpha=\frac{2\mathcal N}{\nu} L<\frac{2\mathcal N}{\nu}
\frac{\nu}{2\mathcal N}=1$, then $\Phi$ is a contraction and,
consequently, there exists a unique function $\overline{\psi}\in
\mathcal L(\mathbb R,\mathfrak{B})$ (respectively, $\gamma \in
C_{b}(\mathbb R,\mathfrak B)$) such that
$\Phi(\overline{\psi})=\overline{\psi}$.

Now we will establish the inequality (\ref{eqBF1}). By Theorem
\ref{thRAPDE1} (item (ii) formula (\ref{eqG2})) we have
\begin{equation}\label{eqES1}
\|\bar{\psi}\|\le \frac{2\mathcal N}{\nu}\sup\limits_{t\in \mathbb
R}|F(t,\bar{\psi}(t)+\varphi_{0}(t))| .
\end{equation}
On the other hand
\begin{equation}\label{eqES2}
|F(t,\bar{\psi}(t)+\varphi_{0}(t))\le
|F(t,\bar{\psi}(t)+\varphi(t))|+|F(t,\varphi_{0}(t))|\le
L|\bar{\psi}(t)|+|F(t,\varphi_{0}(t))|
\end{equation}
and
\begin{equation}\label{eqES3}
|F(t,\varphi_{0}(t))|\le L|\varphi_{0}(t)|\le L\frac{2\mathcal
N}{\nu}\|f\|=\frac{2\mathcal N L}{\nu}\|f\|
\end{equation}
for any $t\in \mathbb R$. From (\ref{eqES1})-(\ref{eqES3}) we
obtain
\begin{equation}\label{eqES4}
\|\bar{\psi}\|\le \frac{2\mathcal
N}{\nu}\big{(}L\|\bar{\psi}\|+\frac{2\mathcal N L}{\nu} \|
f\|\big{)}=\frac{2\mathcal N L}{\nu}\|\bar{\psi}\|
+\frac{4\mathcal N^{2}L}{\nu^{2}}\|f\| \nonumber
\end{equation}
and, consequently,
\begin{equation}\label{eqES5}
\|\bar{\psi}\|\le \frac{4\mathcal N^{2}\|f\|}{\nu (\nu -2\mathcal
N L)}.\nonumber
\end{equation}
To finish the proof of the theorem it is sufficient to assume that
$\varphi:=\overline{\psi}+\varphi_0$ and note that $\varphi$ is
desired solution. The theorem is proved.
\end{proof}

\begin{theorem}\label{th4A.1} \cite{Che_2024.2}
Let $f\in C(\mathbb{R},\mathfrak B)$, $F\in C(\mathbb{R}\times
\mathfrak B,\mathfrak B)$ and $A\in C(\mathbb R,[\mathfrak B])$.

Assume that the following conditions are fulfilled:
\begin{enumerate}
\item[$1)$] the linear equation (\ref{eqLE1}) has the exponential
trichotomy on $\mathbb R$; \item[$2)$] the functions $f$, $A$ and
$F$ are Lagrange stable (respectively, $f$, $A$ and $F$ are
bounded); \item[$3)$] $F$ satisfies the condition of Lipschitz
with respect to $x\in \mathfrak B$ with the constant of Lipschitz
$L<\frac{\nu}{2\mathcal N}$.
\end{enumerate}

Then
\begin{enumerate}
\item the equation $(\ref{eqSL01})$ has a unique Lagrange stable
solution $\varphi \in \mathcal L(\mathfrak L)$ (respectively,
bounded solution $\varphi \in C_{b}(\mathbb R,\mathfrak B),$ where
$\mathfrak L :=\mathfrak L_{(A,f,F)}$; \item the solution
$\varphi$ is defined by equality
\begin{equation}\label{eqBF03}
\varphi(t)=\int_{-\infty}^{+\infty}G_{A}(t,\tau)[f(\tau)+F(\tau,\varphi(\tau))]d\tau
.
\end{equation}
\end{enumerate}
\end{theorem}

\begin{theorem}\label{thLS01} Assume that Condition (\textbf{C}) holds. Under the conditions of Theorem \ref{th4A}
the Lagrange stable solution $\varphi$ defined by (\ref{eqBF03})
is remotely compatible, i.e., $\mathfrak L_{(A,f,F_{Q})}\subseteq
\mathfrak L_{\varphi}$.
\end{theorem}
\begin{proof}
Since the solution $\varphi$ is Lagrange stable, then the set
$Q:=\overline{\varphi(\mathbb R)}$ is compact in $\mathfrak B$.
Let $\{t_n\}\in \mathfrak L_{(A,f,F_{Q})}$, where
$F_{Q}=F\big{|}_{\mathbb R\times Q}$, then the sequences
$\{(A^{t_n},f^{t_n})\}$ and $\{F^{t_n}_{Q}\}$ converge in
$C(\mathbb R,[\mathfrak B])\times C(\mathbb R,\mathfrak B)$ and
$C(\mathbb R\times Q,\mathfrak B)$ respectively. Denote by
$(\tilde{A},\tilde{f})=\lim\limits_{n\to \infty}(A^{t_n},f^{t_n})$
and $\tilde{F}=\lim\limits_{n\to \infty}F_{Q}^{t_n}$. Note that
$\tilde{A}\in \in \Delta_{A}=\alpha_{A}\bigcup \omega_{A}$ and by
Corollary \ref{cor*4.2.3} the equation
\begin{equation}\label{eqDA1}
y'=\tilde{A}(t)y
\end{equation}
is hyperbolic on $\mathbb R$ and, consequently, it has nonzero
bounded on $\mathbb R$ solutions.

Consider the sequence $\{\varphi^{t_n}\}$. It is pre-compact in
$C(\mathbb R,\mathfrak B)$ because $\varphi$ is Lagrange stable.
We need to prove that the sequence $\{\varphi^{t_n}\}$ converges
in $C(\mathbb R,\mathfrak B)$. To this end it is sufficient to
show that it has at most one limiting point. Let $\tilde{\varphi}$
be a limiting point of the sequence $\{\varphi^{t_n}\}$, then
there exists a subsequence $\{t_{n_{k}}\}\subseteq \{t_n\}$ such
that $\tilde{\varphi}=\lim\limits_{k\to
\infty}\varphi^{t_{n_{k}}}$.

Since $\varphi$ is a solution of equation (\ref{eqSL01}), then the
function $\varphi^{t_{n_{k}}}$ is a solution of equation
\begin{equation}\label{eqSL1.01}
x'=\mathcal A^{t_{n_k}}(t) x+f^{t_{n_k}}(t) + F_{Q}^{t_{n_k}}(t,x)
\nonumber
\end{equation}
or equivalently it is a solution of the linear equation
\begin{equation}\label{eqSL1.11}
x'=\mathcal A^{t_{n_k}}(t) x+f^{t_{n_k}}(t) +
F_{Q}^{t_{n_k}}(t,\varphi^{t_{n_k}}(t)).\nonumber
\end{equation}
Denote by $\mathfrak{f}(t):=f(t)+F(t,\varphi(t))$, then
$\mathfrak{f}^{t_{n_{k}}}=f^{t_{n_{k}}}(t)+F_{Q}^{t_{n_{k}}}(t,\varphi^{t_{n_{k}}}(t))$
for any $t\in\mathbb R$. By Lemma \ref{l3.4.2} the function
$\mathfrak{f}\in C(\mathbb R,\mathfrak B)$ is Lagrange stable and
according to Lemma \ref{l3.4.2} (item (ii)) the sequence
$\{\mathfrak{f}^{t_{n_{k}}}\}$ converges in the space $C(\mathbb
R,\mathfrak B)$ to the function
$\tilde{\mathfrak{f}}(t):=\tilde{F}(t,\tilde{\varphi}(t))$ for any
$t\in\mathbb R$. By Theorem \ref{th4A.1} (item (ii)) the function
$\tilde{\varphi}\in C(\mathbb R,\mathfrak B)$ is the unique
compact solution ($\tilde{\varphi(\mathbb R)}\subseteq Q$) of the
equation
\begin{equation}\label{eqSL1.21}
x'=\tilde{\mathcal A}(t)x
+\tilde{f}(t)+\tilde{F}(t,\tilde{\varphi}(t))\nonumber
\end{equation}
or equivalently of the equation
\begin{equation}\label{eqSL1.31}
x'=\tilde{\mathcal{A}}(t)x +\tilde{f}(t)+\tilde{F}(t,x).
\end{equation}
Note that by Theorem \ref{th4A.1} the equation (\ref{eqSL1.31})
has a unique Lagrange stable solution, because the function
$\tilde{F}\in H(F_{Q})$ and by Lemma \ref{lLC1} it is Lipschitzian
with the same constant $L$. Thus every limiting function
$\tilde{\varphi}\in C(\mathbb R,\mathfrak B)$ of the sequence
$\{\varphi^{t_n}\}$ is a Lagrange stable solution of the equation
(\ref{eqSL1.31}). Since the equation (\ref{eqSL1.31}) has a unique
Lagrange stable solution $\tilde{\varphi}\in C(\mathbb R,\mathfrak
B)$, then the sequence $\{\varphi^{t_n}\}$ converges, i.e.,
$\{t_n\}\in \mathfrak L_{\varphi}$. Theorem is proved.
\end{proof}

\begin{coro}\label{corSL1} Assume that Condition (\textbf{C}) holds. Under the conditions of Theorem
\ref{th4A} if the functions $A(t), $$f$  and $F_{Q}$, where
$Q=\overline{\varphi(\mathbb R)}$, are two-sided remotely
stationary (respectively, two-sided remotely $\tau$-periodic,
two-sided remotely almost periodic), then the solution $\varphi$
is so.
\end{coro}
\begin{proof} This statement follows from Theorems \ref{thSL0.1}, \ref{th4A} and \ref{thLS01}.
\end{proof}

\begin{theorem}\label{th4C}
Let $A\in C(\mathbb R,[\mathfrak B])$, $f\in
C(\mathbb{R},\mathfrak B)$ and $F\in C(\mathbb{R}_+\times
\mathfrak B,\mathfrak B)$.

Assume that the following conditions are fulfilled:
\begin{enumerate}
\item[$1)$] the equation (\ref{eqLE1}) is hyperbolic; \item
Condition (\textbf{C}) holds; \item[$2)$] the functions $A$ $f$
and $F$ are Lagrange stable; \item[$3)$] $F$ is Lipschitzian with
respect to the second variable with the constant $L>0$.
\end{enumerate}

Then the following statements hold:
\begin{enumerate}
\item there exists a number $\varepsilon_0>0$ such that for every
$|\varepsilon|\leq \varepsilon_0$ the equation
\begin{equation}\label{eq3.4.6C}
\frac{dx}{dt}=A(t)x+f(t)+\varepsilon F(t,x)
\end{equation}
has a remotely compatible solution $\varphi_{\varepsilon}\in
C_{b}(\mathbb{R},\mathfrak B)$; \item the solution
$\varphi_{\varepsilon}$ is defined by equality
\begin{equation}\label{eqBF04}
\varphi_{\varepsilon}(t)=\int_{-\infty}^{+\infty}G_{A}(t,\tau)[f(\tau)+\varepsilon
F(\tau,\varphi(\tau))]d\tau ;\nonumber
\end{equation}
\item $\{\varphi_{\varepsilon}\}$ converges to $\varphi_0$ as
$\varepsilon\to 0$ uniformly with respect to $t\in\mathbb{R}$,
where $\varphi_0 \in C_{b}(\mathbb R,\mathfrak B)$ is a remotely
compatible solution of the equation (\ref{eqSL01}) defined by the
equality (\ref{eqF1}).
\end{enumerate}
\end{theorem}
\begin{proof} Let $\varepsilon_{0}\in (0,\nu/2\mathcal N L)$.
Since the function $F$ is Lagrange stable, then $\varepsilon F$ is
so. Note the constant of Lipschitz for $\varepsilon F$ is less
than $\varepsilon_{0}L$, because
$$
Lip(\varepsilon F)\le |\varepsilon|Lip(F)\le
\varepsilon_{0}L<\frac{\nu}{2\mathcal N}
$$
for any $|\varepsilon|\le \varepsilon_{0}$. According to Theorem
\ref{thLS01} for every $|\varepsilon|\leq\varepsilon_0$ the
equation (\ref{eq3.4.6C}) has a unique remotely compatible
solution $\varphi_{\varepsilon}\in C(\mathbb R,Q_{r})$ satisfying
the condition $P \varphi_{\varepsilon}(0)=0$.

Let us estimate the difference
$\varphi_{\varepsilon}(t)-\varphi_0(t)=\psi_{\varepsilon}(t)$. By
the inequality (\ref{eqBF1}) we have
\begin{equation}\label{eqBF1.1}
\|\varphi_{\varepsilon}-\varphi_{0}\|\le \frac{4|\varepsilon|
\mathcal N^{2}L\|f\|}{\nu (\nu -2\mathcal N L|\varepsilon|)}
\end{equation}
for any $\varepsilon \in (0,\varepsilon_{0})$. Passing to the
limit in the inequality (\ref{eqBF1.1}) as $\varepsilon\to 0$, we
get the necessary statement. The theorem is proved.
\end{proof}

\begin{remark}\label{remC01} Note that Theorem \ref{th4C}
assures the existence at least one bounded on $\mathbb R$ solution
$\varphi_{\varepsilon}$ of the equation (\ref{eq3.4.6C}) for
sufficiently small $\varepsilon$, but this equation can have on
the space $\mathfrak B$ more than one bounded on $\mathbb R$
solution. This fact we will confirm below by the corresponding
example.
\end{remark}

\begin{example}\label{exC1} Let $p(t):=e^{-|t|}$ for any $t\in \mathbb R$.
Consider the differential equation
\begin{equation}\label{eqC1}
x'=x-\varepsilon p(t)x^{3},
\end{equation}
where $\varepsilon \in \mathbb R_{+}$. For $\varepsilon =0$ admits
a unique bounded on $\mathbb R$ remotely compatible solution
$\varphi_{0}(t)=0$ for all $t\in\mathbb R$. If $\varepsilon
>0$, then the equation (\ref{eqC1}) admits three bounded on $\mathbb R$ solutions:
$\varphi_{\varepsilon}^{1}(t)=0$,
$\varphi_{\varepsilon}^{2}(t)=q_{\varepsilon}(t)$ and
$\varphi_{\varepsilon}^{3}(t)=-q_{\varepsilon}(t)$ for any $t\in
\mathbb R$, where
$$
q_{\varepsilon}(t)=\varepsilon
^{-1/2}\Big{(}2\int\limits_{-\infty}^{t}
e^{-2(t-\tau)}p(\tau)d\tau\Big{)}^{-1/2}\ \ (t\in \mathbb R).
$$
Note that $||\varphi_{\varepsilon}^{1}||\to 0$,
$||\varphi_{\varepsilon}^{2}||\to \infty$ and
$||\varphi_{\varepsilon}^{3}||\to \infty$ as $\varepsilon$ goes to
$0$.
\end{example}

\begin{remark}\label{remKM1} 1. In the case, when the Banach space $\mathfrak
B$ is finite dimensional in the work of Maulen C, Castillo S.,
Kostic M and Pinto M. \cite{MCKP_2021} a result close to results
above was established. Namely, if the functions $A(t)$, $f(t)$ and
$F(t,x)$ are remotely almost periodic on the real axis $\mathbb R$
and the corresponding linear equation (\ref{eqSL1.1}) is
hyperbolic on $\mathbb R$, then under some conditions (of the
"smallness" of $F$) the perturbed equation (\ref{eqSL01}) has a
unique remotely almost periodic on $\mathbb R$ solution.

2. There exist simple examples of the type (\ref{eqSL1.1}) with
remotely almost periodic on $\mathbb R$ coefficients which are
hyperbolic on $\mathbb R_{+}$ and on $\mathbb R_{-}$, but not
hyperbolic on $\mathbb R$. To confirm this statement it is
sufficient to consider the scalar linear differential equation
$$
x'=a(t)x,
$$
where $a(t)=\arctan{t}$ ($t\in\mathbb R$). For the almost periodic
equations this type of examples are impossible because the
exponential dichotomy on $\mathbb R_{+}$ (respectively, on
$\mathbb R_{-}$) implies the exponential dichotomy on the real
axis $\mathbb R$ (for the linear almost periodic equations).
\end{remark}

\section{Funding}

This research was supported by the State Program of the Republic
of Moldova "Monotone Nonautonomous Dynamical Systems
(24.80012.5007.20SE)" and partially was supported by the
Institutional Research Program "SATGED" 011303, Moldova State
University.

\section{Conflict of Interest}

The author declares that he does not have conflict of interest.

\end{document}